\documentclass[a4paper,12pt, twoside, reqno]{amsart}
\usepackage{amssymb}
\usepackage{amsmath}
\newtheorem{theorem}{Theorem}
\newtheorem{corollary}[theorem]{Corollary}

\newtheorem{definition}{Definition}
\newtheorem{example}{Example}
\newtheorem{remark}{Remark}[section]

\title{Krasnoselskij-type algorithms for variational inequality problems and fixed point problems in Banach spaces}
\author{Vasile BERINDE}
\address{Department of Mathematics and Computer Science, Technical University of Cluj-Napoca, North University Center at Baia Mare, 
Victoriei 76, 430122 Baia Mare, ROMANIA; 
E-mail: vasile.berinde@mi.utcluj.ro}

\author{M\u ad\u alina P\u acurar}
\address{Department of Economics and Bussiness Administration
	 in German Language 
	Faculty of Economics  and Bussiness Administration 
	Babe\c s-Bolyai University of Cluj-Napoca 
	T. Mihali 58-60, 400591 Cluj-Napoca, ROMANIA; E-mail: madalina.pacurar@econ.ubbcluj.ro}

\begin{document}
\maketitle \pagestyle{myheadings} \markboth{V. Berinde, M. P\u acurar} {Existence and approximation of fixed points...}
\begin{abstract}
Existence and uniqueness as well as the iterative approximation of fixed points of enriched almost contractions in Banach spaces are studied.  The obtained results are generalizations of the great majority of metric fixed point theorems, in the setting of a Banach space. The main tool used in the investigations is to work with the averaged operator $T_\lambda$ instead of the original operator $T$. The effectiveness of the new results thus derived is illustrated by appropriate examples. An application of the strong convergence theorems to solving a variational inequality is also presented. 
\end{abstract}

\section{Introduction}

Fixed point theory offers important tools for nonlinear analysis in the study of the existence and approximation of the solutions of various nonlinear problems (optimization problems, variational inequality problems, inclusion problems, equilibrium problems, nonlinear functional equations etc.), see for example the comprehensive monograph \cite{Zeid}. The sought solution of such a nonlinear problem is expressed as the  fixed point of a suitable operator, i.e., as the solution of an equivalent fixed point problem,
\begin{equation}\label{fix}
x=Tx,
\end{equation}
where $T$ is defined on a space $X$ endowed with a certain structure. Now, the problem \eqref{fix} is solved by applying an appropriate fixed point theorem and in this way we can immediately construct the Picard iterative sequence defined by
\begin{equation} \label{Picard}
x_{n+1}=Tx_{n},\,n\geq 0,
\end{equation}
where $x_0$ is the initial guess, or a more elaborate fixed point iterative scheme, like the Krasnoselskij iteration, if necessary.

Let us exemplify the above facts in the case of a variational inequality problem. The {\it variational inequality problem} (VIP) for a nonlinear mapping $G:H\rightarrow H$ is defined in the following way:
\begin{equation} \label{VIP}
\textnormal {find } x^*\in C \textnormal { such that } \langle G(x^*), x-x^*\rangle \geq 0, \textnormal{ for all } x\in C,    
\end{equation}
where $C$ is a nonempty convex and closed subset of a real Hilbert space $H$ with inner product and induced norm denoted by $\langle \cdot,\cdot\rangle$ and $\|\cdot \|$, respectively. The variational inequality problem \eqref{VIP}, usually denoted by $VIP(G,C)$, is known (see \cite{Byr}) to be equivalent to the following fixed point problem:
 \begin{equation} \label{vfpp}
x= P_C[x-\lambda G(x)], x\in C,   
\end{equation}
\noindent where $P_C$ is the nearest point projection onto $C$.

It is therefore a challenge, for any nonlinear problem under study,  to find sufficient conditions, as general as possible, to ensure that a certain fixed point iterative sequence $\{x_n\}$ converges to a fixed point of $T$, that is, to a solution of the original nonlinear problem. 

In this paper we introduce a very general class of mappings, called enriched almost contractions, for which we prove existence results, existence and uniqueness results, as well as  convergence theorems for the approximation of their fixed points by means of a Krasnoselskij iterative process.  The obtained results are important generalizations of the great majority of metric fixed point theorems in literature, in the setting of a Banach space, and are applied to obtain strong convergence theorems for the Krasnoselskij projection type iterative method derived from the fixed point formulation of  the variational inequality problem  \eqref{VIP}, i.e., 
 \begin{equation} \label{sequence} 
x_{n+1}=(1-\lambda)x_n+\lambda P_C\left(I-\gamma G\right)x_n,\,n\geq 0,\,x_0\in C.
\end{equation}

To this end, we need some definitions, notations and auxiliary facts in fixed point theory, which are presented in the next section.

\section{Preliminaries} \label{s1}
One of the most useful metrical fixed point theorems in nonlinear analysis is  the famous Banach contraction mapping principle, which is based on the contraction condition 
\begin{equation} \label{contraction}
d(Tx,Ty)\leq c\cdot d(x,y),\,x,y\in X\quad (0\leq c<1).
\end{equation}

The Banach contraction mapping principle 
is a simple and versatile tool in establishing existence and uniqueness theorems for operator equations.  This fact motivated researchers to try to extend and generalize the contraction mapping principle in such a way that its area of applications should be enlarged as much as possible. There is a considerable literature on this topic developed in the last 60 years or so, see for example the monographs  \cite{Aga}, \cite{Ber07}, \cite{Pac09}, \cite{Rus01a} and  \cite{RPP}. 

In \cite{Ber04}, the first author introduced a general class of contractive mappings in metric spaces, called initially {\it weak contractions} and renamed later on as {\it almost contractions}, see \cite{BerP16} for a recent survey on single-valued almost contractions.

\begin{definition} [\cite{Ber04}] \label{def-1}
Let $(X,d)$ be a metric space. A map $T:X\rightarrow X$ is called a $(\delta,L)$-
\textit{almost contraction} if there exist the constants $\delta\in (0,1)$ and $L\geq 0$ such that
\begin{equation}  \label{2.1}
d(Tx, Ty)\leq\delta\cdot d(x,y)+L d(y,Tx)\,, \quad\textnormal{for all}\;\;x,y\in
X.
\end{equation}
\end{definition}

\begin{remark}  Let us denote by $Fix\,(T)$ the set of fixed points of $T$ in $X$, i.e., $Fix\,(T)=\{x\in X: x=Tx\}$. Then $T$ is called {\it a weakly Picard operator} (see Rus \cite{Rus01}, \cite{Rus01a}, \cite{RPP})  if

 $(p1')$ $Fix\,(T)\neq \emptyset$;
 
 $(p2')$ the Picard iteration $\{T^n x\}^\infty_{n=0}$ converges to $x^*\in Fix\,(T)$, for any $x\in X$.
 
 Almost contractions are in general weakly Picard operators, see Theorem \ref{th1} below, which is the main result in \cite{Ber04} and, under an appropriate additional uniqueness condition, they are Picard operators (that is, have a unique fixed point), see Theorem \ref{th2}.

\end{remark}

\begin{theorem} [\cite{Ber04}] \label{th1}
Let $(X,d)$ be a complete metric space and let $T:X\rightarrow X$ be a $(\delta,L)$-almost contraction.

Then

$1)$ $Fix\,(T)=\{x\in X:Tx=x\}\neq\emptyset$;

$2)$ For any $x_0\in X$, the Picard iteration
$\{x_n\}^\infty_{n=0}$, $x_n=T^n x_0$,  converges to some $x^\ast\in Fix\,(T)$;

$3)$ The following estimate holds
\begin{equation}  \label{3.2-1}
d(x_{n+i-1},x^\ast) \leq\frac{\delta^i}{1-\delta}\, d(x_n,
x_{n-1})\,,\quad n=0,1,2,\dots;\,i=1,2,\dots
\end{equation}
\end{theorem}

\begin{theorem} [\cite{Ber04}] \label{th2}
Let $(X,d)$ be a complete metric space and $T:X\rightarrow X$ be a $(\delta,L)$- almost contraction for which there exist $\delta_1 \in (0,1)$ and $L_1\geq 0$
such that
\begin{equation} \label{3.8}
d(Tx,Ty)\leq \delta_1 \cdot d(x,y)+L_1 \cdot d(x,Tx)\,,\quad \textnormal{for all}%
\;\;x,y\in X\,.  
\end{equation}
Then

$1)$ $T$ has a unique fixed point, i.e. $Fix\,(T)=\{x^{\ast }\}$;

$2)$ The Picard iteration $\{x_{n}\}_{n=0}^{\infty }$ associated to $T$
converges to $x^{\ast }$, for any $x_{0}\in X$;

$3)$ The  estimate \eqref{3.2-1} holds.

$4)$ The rate of convergence of the Picard iteration is given by
\begin{equation}
d(x_{n},x^{\ast })\leq \delta \,d(x_{n-1},x^{\ast })\,,\quad n=1,2,\dots
\label{3.10}
\end{equation}
\end{theorem}

As shown in \cite{Ber04}, the class of almost contractions is very large and includes most of the main important contractive type mappings in metrical fixed point theory, according to Rhoades's classification: Picard-Banach contractions, Kannan contractions \cite{Kan68}, \cite{Kan69}, Chatterjea mappings \cite{Chat}, as well as Zamfirescu operators \cite{Zam} etc. It also includes some nonexpansive mappings, as shown by the next example.

\begin{example} [\cite{Ber04}] \label{ex3}
Let $[0,1]$ be the unit interval with the usual norm. Let $T:[0,1]\rightarrow [0,1]$ be the identity map, i.e., $Tx=x$,
for all $x\in [0,1]$.\newline
Then 

1) $T$ satisfies the almost contraction condition (\ref{2.1}) with $\delta\in(0,1)$ arbitrary and $%
L\geq 1-\delta$. 

2) $Fix\,(T)=[0,1]$. 
\end{example}

On the other hand, using the technique of enrichment of nonexpansive mappings, introduced by the first author in \cite{Ber19b}, in the case of Hilbert spaces, the current authors introduced and studied the following classes of enriched mappings: enriched strictly pseudocontractive operators in Hilbert spaces \cite{Ber18}, enriched nonexpansive mappings in Banach spaces \cite{Ber20}, enriched Picard-Banach contractions \cite{BerP20}, enriched Kannan mappings  \cite{BerP21} and enriched Chatterjea mappings  \cite{BerP19c}, all in the setting of a real Banach space. 

\section{Enriched  almost contractions in Banach spaces} \label{s2}

It is therefore natural to introduce and study the class of enriched almost contractions, which will include all the previous mentioned classes of enriched contractive type mappings. Due to the way the enrichment process of an operator $T$ is performed, that is, by means of the averaged operator $T_\lambda$, we work in the setting of a Banach space. We start by presenting the first main result from \cite{BerP20}.

\begin{definition}[\cite{BerP20}] \label{def0}
Let $(X,\|\cdot\|)$ be a linear normed space. A mapping $T:X\rightarrow X$ is said to be a $(b,\theta$)-{\it enriched contraction} if there exist $b\in[0,+\infty)$ and $\theta\in[0,b+1)$ such that
\begin{equation} \label{eq3}
\|b(x-y)+Tx-Ty\|\leq \theta \|x-y\|,\forall x,y \in X.
\end{equation}
\end{definition}

\begin{theorem} [\cite{BerP20}] \label{th0}
Let $(X,\|\cdot\|)$ be a Banach space and $T:X\rightarrow X$ a $(b,\theta$)-{\it enriched contraction}. Then

$(i)$ $Fix\,(T)=\{p\}$, for some $p\in X$;

$(ii)$ There exists $\lambda\in (0,1]$ such that the iterative method
$\{x_n\}^\infty_{n=0}$, given by
\begin{equation}
x_{n+1}=(1-\lambda)x_n+\lambda T x_n,\,n\geq 0,
\end{equation}
converges to p, for any $x_0\in X$;

$(iii)$ The  estimate \eqref{3.2-1} holds with $\delta=\dfrac{\theta}{b+1}$.
\end{theorem}

Our aim in this section is to generalize the class of mappings introduced in Definition \ref{def0} and the corresponding fixed point result stated in Theorem \ref{th0}.

\begin{definition}\label{def1}
Let $(X,\|\cdot\|)$ be a linear normed space. A mapping $T:X\rightarrow X$ is said to be an {\it enriched almost contraction} if there exist $b\in[0,\infty)$, $\theta\in (0,b+1)$ and $L\geq 0$ such that
\begin{equation} \label{eq3}
\|b(x-y)+Tx-Ty\|\leq \theta \|x-y\|+L\|b(x-y)+Tx-y\|,
\end{equation}
for all $x,y \in X$.
To indicate the constants involved in \eqref{eq3} we shall also call $T$ as an  {\it enriched} $(b,\theta, L)$-{\it  almost contraction}. 
\end{definition}

\begin{example}[ ] \label{ex1}
1) Any $(\delta,L)$-almost contraction in the sense of Definition \ref{def-1} is an enriched $(0,\delta,L)$-almost contraction, i.e., it satisfies \eqref{eq3} with $b=0$ and $\theta=\delta$;

2) Any $(b,\theta)$-enriched  contraction in the sense of Definition \ref{def0},  is an enriched $\left(b,\theta, 0\right)$-almost contraction;

3) Any  $(k,a$)-{\it enriched Kannan mapping} \cite{BerP21}, that is, any mapping $T:X\rightarrow X$ for which there exist  $a\in[0,1/2)$ and $k\in[0,\infty)$ such that
\begin{equation} \label{cond_eKannan}
\|k(x-y)+Tx-Ty\|\leq a \left(\|x-Tx\|+\|y-Ty\|\right),\ for\ all\ x,y \in X,
\end{equation}
 is an enriched $\left(k, \dfrac{a}{1-a}, \dfrac{2 a}{1-a}\right)$-almost contraction;

4) Any $(k,b$)-{\it enriched Chatterjea mapping} \cite{BerP19c}, that is, 
any mapping $T:X\rightarrow X$ for which there exist $b\in[0,1/2)$ and $k\in[0,+\infty)$  such that
$$
\|k(x-y)+Tx-Ty\|\leq  b \left[\|(k+1)(x-y)+y-Ty\|+\right.
$$
\begin{equation} \label{Def_eChatterjea}
\left.+\|(k+1)(y-x)+x-Tx\|\right]\|,\forall x,y \in X,
\end{equation}
 is an enriched $\left(k, \dfrac{b}{1-b}, \dfrac{2 b}{1-b}\right)$-almost contraction.
\end{example}

For many other examples of enriched almost contraction we refer to \cite{BerP16}, where more than 270 developments around the concept of almost contraction are included in the list of References.


Despite the fact that the fixed points of almost contractions can be obtained by means of the Picard iteration, see Theorems \ref{th1} and \ref{th2}, in the case of enriched almost contractions the situation is changed, as the Picard iteration might be not convergent, see Example \ref{ex2a}, or, even if it would be convergent, the limit of the sequence generated by the Picard iteration associated to $T$ could converge to an element which is not  a fixed point of $T$.

This is the reason why in the next theorem we use the Krasnoselskij iteration process to approximate the fixed points of an enriched almost contraction.

\begin{theorem}  \label{th3}
Let $(X,\|\cdot\|)$ be a Banach space and let $T:X\rightarrow X$ be a $(b,\theta,L)$-almost contraction.

Then

$1)$ $Fix\,(T)\neq\emptyset$;

$2)$ For any $x_0\in X$, there exists $\lambda\in (0,1)$ such that the Krasnoselskij iteration 
$\{x_n\}^\infty_{n=0}$, defined by,  
\begin{equation} \label{eq3a}
x_{n+1}=(1-\lambda)x_n+\lambda T x_n,\,n\geq 0,
\end{equation}
converges to some $x^\ast\in Fix\,(T)$, for any $x_0\in X$;

$3)$ The  estimate \eqref{3.2-1} holds
with $\delta=\dfrac{\theta} {b+1}$.
\end{theorem}

\begin{proof}
For any $\lambda\in (0,1)$ consider the averaged mapping $T_\lambda$, given by
\begin{equation} \label{eq4}
T_\lambda (x)=(1-\lambda)x+\lambda T(x), \forall  x \in X.
\end{equation}
It is well known that $T_\lambda$ possesses the following important property:
$$
Fix(\,T_\lambda)=Fix\,(T).
$$
Now, since $T$ is a $(b,\delta,L)$-almost contraction, there exist $b\in[0,\infty)$, $\theta\in (0,b+1)$ and $L\geq 0$ such that
\begin{equation} \label{eqEAC}
\|b(x-y)+Tx-Ty\|\leq \theta \|x-y\|+L\|b(x-y)+Tx-y\|,
\end{equation}
for all $x,y \in X$.

If $b=0$, the conclusion follows by Theorem \ref{th2}.

If $b> 0$ in \eqref{eqEAC}, then let us put  $\lambda=\dfrac{1}{b+1}$. Obviously, we have $0<\lambda<1$ and thus the contractive condition \eqref{eqEAC} becomes
$$
\left \|\left(\frac{1}{\lambda}-1\right)(x-y)+Tx-Ty\right\|\leq \theta \|x-y\|
$$
$$
+L\left \|\left(\frac{1}{\lambda}-1\right)(x-y)+Tx-y\right\|,\forall x,y \in X,
$$
which, after some simple computations, can be written in an equivalent form
\begin{equation} \label{eq5}
\|T_\lambda x-T_\lambda y\|\leq \delta \|x-y\|+L \|T_\lambda x-y\| ,\forall x,y \in X,
\end{equation}
where $\delta=\dfrac{\theta} {b+1}\in (0,1)$.

The above inequality shows that $T_\lambda$ is an almost contraction in the Banach space $X$. We shall write in the following the norm in $X$ as the induced distance, i.e., 
$$\|x-y\|=d(x,y),$$ 
in order to emphasize the fact that the rest of the proof  works in the general case of a complete metric space.

We shall prove that $T_\lambda$ satisfying \eqref{eq5} has at least a fixed point in $X$. To this
end, let $x_0\in X$ be arbitrary but fixed and let $\{x_n\}^\infty_{n=0}$ be the
Picard iteration defined by 
\begin{equation}\label{t-lambda}
x_{n+1}=T_\lambda x_n,\, n\geq 0.
\end{equation}

Take $x:=x_{n-1}$, $y:=x_n$ in \eqref{eq5} to obtain
$$  d(T_\lambda x_{n-1}, T_\lambda x_n)\leq \delta\cdot d(x_{n-1}, x_n)\,,  $$
which shows that
\begin{equation}\label{3.3}
    d(x_n, x_{n+1})\leq \delta\cdot d(x_{n-1}, x_n)\,.
\end{equation}
Using (\ref{3.3}) we obtain by induction
$$  d(x_n, x_{n+1})\leq \delta^n d(x_0, x_1)\,,
    \quad n=0,1,2,\dots
$$
and then
\begin{eqnarray}\label{3.4}
  d(x_n, x_{n+p}) &\leq & \delta^n
  \big(1+\delta+\cdots+\delta^{p-1}\big) d(x_0, x_1)=\nonumber\\
   &=& \frac{\delta^n}{1-\delta}\,(1-\delta^p)\cdot d(x_0, x_1)\,,
   \quad n,p\in\mathbb{N},\;p\neq 0\,.\qquad
\end{eqnarray}
Since $0<\delta<1$, (\ref{3.4}) shows that $\{x_n\}^\infty_{n=0}$ is a
Cauchy sequence and hence it is convergent.
Let us denote
\begin{equation}\label{3.5}
    x^\ast=\lim_{n\to\infty} x_n\,.
\end{equation}
Then
$$  d(x^\ast, T_\lambda x^\ast)\leq d(x^\ast, x_{n+1})+
    d(x_{n+1}, T_\lambda x^\ast)=d(x_{n+1}, x^\ast)+
    d(T_\lambda x_n, T_\lambda x^\ast)\,.
$$
By (\ref{2.1}) we have
$$  d(T_\lambda x_n, T_\lambda x^\ast)\leq\delta\, d(x_n, x^\ast)+L\,d(x^\ast, T_\lambda x_n)  $$
and hence, by combining the previous two inequalities, one obtains
\begin{equation}\label{3.6}
    d(x^\ast, T_\lambda x^\ast)\leq (1+L) d(x^\ast, x_{n+1})+
    \delta\cdot d(x_n, x^\ast)\,,
\end{equation}
which is valid for all $n\geq 0$. Letting $n\to\infty$ in (\ref{3.6}) we obtain
$$  d(x^\ast, T_\lambda x^\ast)=0  $$
which shows that $x^\ast$ is a fixed point of $T_\lambda$. 
As 
$$Fix\,(T_\lambda)=Fix\,(T),$$
 the first part of the theorem follows.
Now, by letting  $p\to\infty$ in (\ref{3.4}), one obtains the \emph{a priori} estimate
\begin{equation}  \label{3.1}
d(x_n, x^\ast)\leq\frac{\delta^n}{1-\delta}\, d(x_0, x_1)\,,\quad
n=0,1,2,\dots,
\end{equation}
where $\delta=\dfrac{\theta} {b+1}\in (0,1)$.

On the other hand, let us observe that by (\ref{3.3}) we inductively
obtain
$$  d(x_{n+k}, x_{n+k+1})\leq \delta^{k+1}\cdot
    d(x_{n-1}, x_n)\,,\quad k,n\in\mathbb{N}\,,
$$
and hence, similarly to deriving (\ref{3.4}), we get
\begin{equation}\label{3.7}
    d(x_n, x_{n+p})\leq\frac{\delta(1-\delta^p)}{1-\delta}\,
    d(x_{n-1}, x_n)\,,\quad n\geq 1,\; p\in\mathbb{N}^\ast\,.
\end{equation}
Now, by letting $p\to\infty$ in (\ref{3.7}), the following \emph{a posteriori} estimate follows:
\begin{equation}  \label{3.2}
d(x_n, x^\ast)\leq\frac{\delta}{1-\delta}\, d(x_{n-1}, x_n)\,,\quad
n=1,2,\dots
\end{equation}
Obviously, (\ref{3.1}) and (\ref{3.2}) could be merged to obtain the desired unified estimate. 
\end{proof}
\begin{remark}

 Theorem \ref{th3} is an effective generalization of Theorems \ref{th1}, \ref{th2} and \ref{th0} above, of Theorems 2.1 and 3.1 in \cite{BerP21} and of Theorems 2.4 and 3.2 in \cite{BerP19c},  as shown by the next example. 

\end{remark}

\begin{example}[ ] \label{ex2}

Let $X=\left[0,\frac{4}{3}\right]$ with the usual norm and $T: X\rightarrow X$ be given by
\begin{equation} \label{eac}
Tx=
\begin{cases}
1-x, \textnormal{ if } x\in \left[0,\frac{2}{3}\right)\\
\smallskip
2-x, \textnormal{ if } x\in \left[\frac{2}{3}, \frac{4}{3}\right].
\end{cases}
\end{equation}
Then $Fix\,(T)=\left\{\dfrac{1}{2},1\right\}$ and:

1) $T$ is a $(1,\theta, 3)$-enriched almost contraction, for any $\theta\in (0,2)$;

2) $T$ is not an almost contraction;

3) $T$ does not belong to the classes of enriched contractions (\cite{BerP20}), enriched Kannan mappings (\cite{BerP21}) or enriched Chatterjea mappings (\cite{BerP19c});

4) $T$ is neither nonexpansive nor quasi-nonexpansive;

5)  $T$ is not an enriched nonexpansive mapping (\cite{Ber19b}).

\end{example}

\begin{proof}
1) The enriched almost contraction condition \eqref{eq3} is in this case
\begin{equation} \label{eq100}
|b(x-y)+Tx-Ty|\leq \theta |x-y|+L\cdot |b(x-y)+Tx-y|, x,y\in \left[0,\frac{4}{3}\right],
\end{equation}
where $b\in [0,\infty)$, $\theta \in (0,b+1)$ and $L\geq 0$ are constants to be shown that exist.

Because the contraction condition \eqref{eq100} is not symmetric, we have to discuss the following four cases:

{\bf Case 1a}. $x,y\in \left[0,\dfrac{2}{3}\right)$
\smallskip

In this case, condition \eqref{eq100} is satisfied with $b=1$, $\theta\in (0,2)$ arbitrary and $L=0$;
\smallskip

{\bf Case 1b}. $x,y\in \left[\dfrac{2}{3},\dfrac{4}{3}\right]$
\smallskip

Similarly, condition \eqref{eq100} is satisfied if we take $b=1$, $\theta\in (0,2)$ arbitrary and $L=0$;
\smallskip

{\bf Case 1c}. $x\in \left[0,\dfrac{2}{3}\right)$, $y\in \left[\dfrac{2}{3},\dfrac{4}{3}\right]$
\smallskip

In this case, $Tx=1-x$, $Ty=2-y$ and hence condition \eqref{eq100} becomes
\begin{equation}\label{eq10a}
|(b-1)(x-y)-1|\leq \theta |x-y|+L\cdot |(b-1)(x-y)+1-2y|.
\end{equation}
By letting $b=1$ in \eqref{eq10a} we get
\begin{equation}\label{eq11a}
1\leq \theta |x-y|+L\cdot |1-2y|, x\in \left[0,\dfrac{2}{3}\right), y\in \left[\dfrac{2}{3},\frac{4}{3}\right].
\end{equation}
Now, for $y\in \left[\dfrac{2}{3},\dfrac{4}{3}\right]$, we have $|1-2y|\geq \dfrac{1}{3}$ and therefore, in order to have \eqref{eq11a} satisfied for all $x\in \left[0,\dfrac{2}{3}\right)$ and $y\in \left[\dfrac{2}{3},\dfrac{4}{3}\right]$, it is sufficient to choose $L\geq 3$. 

So, if we take $b=1$, $\theta\in (0,2)$ arbitrary and $L\geq 3$, inequality \eqref{eq10a} holds true for all $x\in \left[0,\dfrac{2}{3}\right)$ and $y\in \left[\dfrac{2}{3},\dfrac{4}{3}\right]$.

\smallskip

{\bf Case 1d}. $x\in \left[\dfrac{2}{3},\dfrac{4}{3}\right]$, $y\in \left[0,\dfrac{2}{3}\right)$
\smallskip

In this case, $Tx=2-x$, $Ty=1-y$ and hence condition \eqref{eq100} becomes
\begin{equation}\label{eq12a}
|(b-1)(x-y)+1|\leq \theta |x-y|+L\cdot |(b-1)(x-y)+2-2y|.
\end{equation}
By letting $b=1$ in \eqref{eq12a} we get
\begin{equation}\label{eq13a}
1\leq \theta |x-y|+L\cdot |2-2y|, x\in \left[0,\dfrac{2}{3}\right), y\in \left[\dfrac{2}{3},\frac{4}{3}\right].
\end{equation}
Now, for $y\in \left[0,\dfrac{2}{3}\right)$ we have $|2-2y|\geq \dfrac{2}{3}$ and therefore, in order to have  \eqref{eq13a} satisfied, for all $x\in \left[\dfrac{2}{3},\dfrac{4}{3}\right]$, $y\in \left[0,\dfrac{2}{3}\right)$, it is enough to take $L\geq \dfrac{3}{2}$. 

So, \eqref{eq12a} holds true for all $x\in \left[\dfrac{2}{3},\dfrac{4}{3}\right]$, $y\in \left[0,\dfrac{2}{3}\right)$ if we choose $b=1$, $\theta\in (0,2)$ arbitrary and $L\geq \dfrac{3}{2}$.

By summarizing all four cases, we conclude that with $b=1$, $\theta\in (0,2)$ arbitrary and $L=3$, inequality \eqref{eq100} holds true for all $x,y\in\left[0,\dfrac{4}{3}\right]$, that is,
$T$ is a $(1,\theta, 3)$-enriched almost contraction, with $\theta\in (0,2)$ arbitrary.

2) To show that $T$ is not an almost contraction, just let $x=\dfrac{7}{15}$ and  $y=\dfrac{8}{15}$ in the almost contraction condition \eqref{2.1} to obtain (we recall that $\delta <1$):
$$
\frac{1}{15}\leq \delta \cdot \frac{1}{15}<\frac{1}{15},
$$
a contradiction.

3) $T$ in our example has two fixed points, i.e., $Fix\,(T)=\left\{\dfrac{1}{2},1\right\}$, hence it is not an enriched contraction (\cite{BerP20}), since any enriched contraction has a unique fixed point, see Theorem \ref{th0}. A similar motivation works for the other two classes of contractive mappings (enriched Kannan mappings (\cite{BerP21}) and enriched Chatterjea mappings (\cite{BerP19c})).

4) To see that $T$ is not nonexpansive, we note that $T$ is not continuous at $x=\dfrac{2}{3}$. Assume now that $T$ is quasi-nonexpansive. This would mean that
$$
|Tx-x^*|\leq |x-x^*|,\, \forall x\in \left[0,\dfrac{4}{3}\right], x^*\in Fix\,(T).
$$
Just take $x=\dfrac{2}{3}$ and $x^*=\dfrac{1}{2}$ to get $\dfrac{5}{6}\leq \dfrac{1}{6}$, a contradiction.

5) According to (\cite{Ber19b}, Definition 2.1), $T$ is an enriched nonexpansive mapping if there exists $b\in[0,\infty)$  such that
\begin{equation} \label{eq33a}
\|b(x-y)+Tx-Ty\|\leq (b+1) \|x-y\|,\forall x,y \in X.
\end{equation}
In the particular case of this example, by letting $x=\dfrac{2}{3}$ and $y=\dfrac{1}{2}$ in \eqref{eq33a} we get
$$
(b+1)\cdot \frac{1}{2}\leq (b+1)\cdot \frac{1}{6},
$$
a contradiction.
\end{proof}

\begin{remark}\label{rem2.3}
In connection to item 5) in the above example, we also note that there exists mappings which are simultaneously enriched nonexpansive and $(b,\theta, L)$-enriched almost contractions. The simplest example is the identity map on $[0,1]$, see Example \ref{ex3}.
\end{remark}

\begin{example}[ ] \label{ex2a}
Let $T$ be the enriched almost contraction defined in Example \ref{ex2}. Then, the Picard iteration associated to $T$, that is, the sequence $x_{n+1}=T x_n$, $n\geq 0$, starting from any $x_0\in \left[0, \dfrac{2}{3}\right)\setminus \left\{\dfrac{1}{2}\right\}$ does not converge. However, the Krasnoselskij iteration 
$$
y_{n+1}=(1-\lambda)y_n+\lambda T y_n,\,n\geq 0,
$$
converges to $\dfrac{1}{2}\in Fix\,(T)$, for any $y_0\in \left[0, \dfrac{2}{3}\right)$ and any $\lambda \in (0,1)$, which follows immediately by computing $y_n$ in closed form:
$$
y_n=\dfrac{1}{2}+(1-2\lambda)^n\cdot (y_0-\dfrac{1}{2}),\, n\geq 0.
$$

Similarly, the Picard iteration associated to $T$, that is, the sequence $x_{n+1}=T x_n$, $n\geq 0$, starting from any $x_0\in \left[\dfrac{2}{3}, \dfrac{4}{3}\right]\setminus \left\{1\right\}$ diverges but the Krasnoselskij iteration 
$$
y_{n+1}=(1-\lambda)y_n+\lambda T y_n,\,n\geq 0,
$$
converges to $1\in Fix\,(T)$, for any $y_0\in \left[\dfrac{2}{3}, \dfrac{4}{3}\right]$  and any $\lambda \in (0,1)$.
\end{example}

Now we state an existence and uniqueness result for enriched $(b,\theta,L)$-almost contractions which corresponds to and generalizes Theorem \ref{th2}.

\begin{theorem} \label{th2a}
Let $(X,\|\cdot\|)$ be a Banach space and $T:X\rightarrow X$ be an enriched $(b,\theta,L)$-almost contraction for which there exist $\delta_1 \in (0,1)$ and $L_1\geq 0$
such that
\begin{equation} \label{3.8}
\|Tx-Ty\|\leq \delta_1 \cdot \|x-y\|+L_1 \cdot \|x-Tx\|\,,\quad \textnormal{for all}%
\;\;x,y\in X.
\end{equation}
Then

$1)$ $T$ has a unique fixed point, i.e., $Fix\,(T)=\{x^{\ast }\}$;

$2)$ For any $x_0\in X$, there exists $\lambda\in (0,1)$ such that the Krasnoselskij iteration 
$\{x_n\}^\infty_{n=0}$, defined by,  
$$
x_{n+1}=(1-\lambda)x_n+\lambda T x_n,\,n\geq 0,
$$
converges strongly to $x^\ast$, for any $x_0\in X$;

$3)$ The  estimate \eqref{3.2-1} holds
with $\delta=\dfrac{\theta} {b+1}$.

$4)$ The rate of convergence of the Krasnoselskij iteration is given by
$$
\|x_{n}-x^{\ast }\|\leq \delta \,\|x_{n-1}-x^{\ast }\|\,,\quad n=1,2,\dots
$$
\end{theorem}

\begin{proof}
The existence of the fixed point, the convergence of the Krasnoselskij iteration as well as the error estimate follows by Theorem \ref{th3}. 

Suppose $T$ has two distinct fixed points, $x^*,y^*$. Then, by taking $x:=x^*$, $y:=y^*$ in \eqref{3.8}, we obtain
$$
d(x^*,y^*)\leq \delta_1 \cdot d(x^*,y^*),
$$
which leads to $d(x^*,y^*)=0$, that is, $x^*=y^*$, a contradiction.
\end{proof}

\begin{remark}

1) Theorem \ref{th2} is a corollary of Theorem \ref{th2a} and is obtained for $b=0$ and $\theta=\delta$ (see Example \ref{ex1}, item 1). 

2) Theorem \ref{th0} is a corollary of Theorem \ref{th2a} and is obtained for $L=0$.  

\end{remark}

We end this section by stating some other important fixed point results which are corollaries of  Theorem \ref{th2a}.

\begin{corollary} [\cite{BerP21}, Theorem 2.1] \label{cor1}
Let $(X,\|\cdot\|)$ be a Banach space and $T:X\rightarrow X$ a $(k,a$)-{\it enriched Kannan mapping}. 
Then

$(i)$ $Fix\,(T)=\{p\}$, for some $p\in X$;

$(ii)$ There exists $\lambda\in (0,1]$ such that the iterative method
$\{x_n\}^\infty_{n=0}$, given by
\begin{equation} 
x_{n+1}=(1-\lambda)x_n+\lambda T x_n,\,n\geq 0,
\end{equation}
converges to p, for any $x_0\in X$;

$(iii)$ The  estimate \eqref{3.2-1} holds
with $\delta=\dfrac{a}{1-a}$.
\end{corollary}

\begin{proof}
By Example \ref{ex1}, item 2, any  $(k,a$)-{\it enriched Kannan mapping}
 is an enriched $\left(k, \dfrac{a}{1-a}, \dfrac{2 a}{1-a}\right)$-almost contraction. Conclusion follows by  Theorem \ref{th2a}.
\end{proof}

\begin{corollary} [\cite{BerP21}, Theorem 3.1]\label{cor2}
Let $(X,\|\cdot\|)$ be a Banach space and $T:X\rightarrow X$ a $(k,h$)-{\it enriched Bianchini mapping}, that is, a mapping for which there exist  $h\in[0,1)$ and $k\in[0,\infty)$ such that
\begin{equation} \label{Def_eBianchini}
\|k(x-y)+Tx-Ty\|\leq h \max\left\{\|x-Tx\|,\|y-Ty\|\right\},\textnormal{ for all } x,y \in X.
\end{equation}. 
Then

$(i)$ $Fix\,(T)=\{p\}$, for some $p\in X$;

$(ii)$ There exists $\lambda\in (0,1]$ such that the iterative method
$\{x_n\}^\infty_{n=0}$, given by
\begin{equation} \label{eq1.3a}
x_{n+1}=(1-\lambda)x_n+\lambda T x_n,\,n\geq 0,
\end{equation}
converges to p, for any $x_0\in X$;

$(iii)$ The following estimate holds
\begin{equation}  \label{1.3b}
\|x_{n+i-1}-p\| \leq\frac{h^i}{1-h}\cdot \|x_n-
x_{n-1}\|\,,\quad n=0,1,2,\dots;\,i=1,2,\dots
\end{equation}
\end{corollary}

\begin{proof}
Similarly to Example \ref{ex1}, item 2, we can easily prove that any  $(k,h$)-{\it enriched Bianchini mapping} 
 is an enriched $\left(k, h, 2h\right)$-almost contraction. Conclusion follows by  Theorem \ref{th2a}.
\end{proof}

\begin{corollary}[\cite{BerP19c}, Theorem 2.4] \label{cor3}
Let $(X,\|\cdot\|)$ be a Banach space and $T:X\rightarrow X$ a $(k,b$)-{\it enriched Chatterjea mapping}. Then

$(i)$ $Fix\,(T)=\{p\}$, for some $p\in X$;

$(ii)$ There exists $\lambda\in (0,1]$ such that the iterative method
$\{x_n\}^\infty_{n=0}$, given by
$$
x_{n+1}=(1-\lambda)x_n+\lambda T x_n,\,n\geq 0,
$$
converges to p, for any $x_0\in X$;

$(iii)$ The following estimate holds
$$
\|x_{n+i-1}-p\| \leq\frac{\delta^i}{1-\delta}\cdot \|x_n-
x_{n-1}\|\,,\quad n=0,1,2,\dots;\,i=1,2,\dots
$$
where $\delta=\dfrac{b}{1-b}$.
\end{corollary}

\begin{proof}
By Example \ref{ex1}, item 3, any  $(k,b$)-{\it enriched Chatterjea mapping} 
 is an enriched $\left(k, \dfrac{b}{1-b}, \dfrac{2 b}{1-b}\right)$-almost contraction. Conclusion follows by  Theorem \ref{th2a}.
\end{proof}

\section{An application to variational inequality problems} \label{s3}

Let $H$ be a Hilbert space and let $C\subset H$ be closed and convex. A mapping $G:H\rightarrow H$ is called {\it monotone} if
$$
\langle Gx-Gy,x-y\rangle \geq 0, \textnormal{ for all } x,y\in H.
$$
The {\it variational inequality problem} with respect to $G$ and $C$, denoted by $VIP(G,C)$, is to find $x^*\in C$ such that
$$
\langle Gx^*,x-x^*\rangle \geq 0, \textnormal{ for all } x\in H.
$$
It is well known, see for example \cite{Byr}, that if $\gamma>0$ then $x^*\in C$ is a solution of the $VIP(G,C)$ if and only if $x^*$ is a solution of the fixed point problem
$$
x=P_C\left(I-\gamma G\right) x,
$$
where $I$ is the identity mapping on $H$ and $P_C$ is the nearest point projection onto $C$. 

In Byrne \cite{Byr} it is proven, amongst many other important results, that, if $I-\gamma G$ and $P_C\left(I-\gamma G\right)$ are averaged nonexpansive mappings, then, under some additional assumptions, the iterative algorithm $\{x_n\}$ defined by 
$$
x_{n+1}=P_C\left(I-\gamma G\right) x_n,\,n\geq 0,
$$
converges {\it weakly} to a solution of the $VIP(G,C)$, if such solutions exist.

Our alternative here is to consider $VIP(G,C)$ for enriched almost contractions, which are in general discontinuous mappings, instead of nonexpansive mappings, which are always continuous. 

In this case we shall have $VIP(G,C)$  with a non unique solution, as shown by the next theorem and the considered algorithm \eqref{a7} will convergence strongly to the solution of the $VIP(G,C)$.

\begin{theorem} \label{th4}
Assume that for $\gamma >0$ and $P_C\left(I-\gamma G\right)$ is a $(k,\theta, L)$-enriched almost contraction. Then there exists $\lambda\in(0,1]$ such that the iterative algorithm $\{x_n\}$ defined by 
\begin{equation}\label{a7}
x_{n+1}=(1-\lambda)x_n+\lambda P_C\left(I-\gamma G\right)x_n,\,n\geq 0,
\end{equation}
converges strongly to a solution $x^*$ of the $VIP(G,C)$, for any $x_0\in C.$
\end{theorem}

\begin{proof}
Since $C$ is closed, we take $X:=C$ and $$T:=P_C\left(I-\gamma A^*\left(I-P_Q\right)A\right)$$ and apply Theorem \ref{th3}.
\end{proof}
From Theorem \ref{th4} we obtain as a corollary Theorem 4.2 in \cite{BerP21}.

\begin{remark}
Note that in case of Theorem \ref{th4}, the $VIP(G,C)$ has no unique solution, while in the case of Theorem 4.2 in \cite{BerP21}, the solution is unique.

\end{remark}
\begin{example}
Let $T=P_C\left(I-\gamma G\right)$ in Theorem \ref{th4} be the function defined in Example \ref{ex2}. Then for $C=\left[\dfrac{2}{3}, \dfrac{4}{3}\right]$, $T$ is nonexpansive.  Although for (averaged) nonexpansive mappings one can guarantee only {\it weak convergence} of the iterative process, see \cite{Byr}, Theorem \ref{th4} has the remarkable feature that it ensures the {\it strong convergence} of the iterative process \eqref{a7} to the solution $x^*=1$. 


\end{example}

\section{Conclusions and an open problem}

\quad

1. In this paper we introduced and studied in the setting of a Banach space a very general class of contractive type mappings, called {\it enriched almost contractions}, which includes most of the contractive mappings in Rhoades' classification \cite{Rho}. We mention a few of them:
 Banach contractions \cite{Ban22}; Kannan contractions \cite{Kan68}, \cite{Kan69};
 Chatterjea contractions \cite{Chat};
 Zamfirescu contractions \cite{Zam}.

The class of enriched almost contractions also includes:
 almost contractions \cite{Ber04}, \cite{BerP16};
 enriched contractions \cite{BerP20};
 enriched Kannan mappings and enriched Bianchini mappings  \cite{BerP21}; enriched Chatterjea mappings \cite{BerP19c};
 enriched Zamfirescu contractions and many others.

For all previous classes, the inclusion is effective, as shown by  Example \ref{ex1}.
 \smallskip
 
2. We have shown (Theorem \ref{th3}) that, for any enriched almost contraction, the set of its fixed points is nonempty and, in general, this set contains more than one element. Moreover, there exist enriched almost contractions with the set of fixed points an infinite set, see the example mentioned in Remark \ref{rem2.3}.
\smallskip

3. We have shown (Theorem \ref{th3}) that the fixed points of an enriched almost contraction can be approximated by means of an appropriate Kransnoselskij iteration and that, in general, the Picard iteration cannot be used to approximate the fixed points of an enriched almost contraction, see Example \ref{ex2}. 
\smallskip

4. We also presented an existence and uniqueness result (Theorem \ref{th2a}) for enriched almost contractions, by imposing an additional condition. Various fixed point theorems for contractive mappings in Rhoades' classification as well as some convergence theorems for enriched mappings are particular cases of  Theorem \ref{th2a}. Amongst them we mention Theorem \ref{th0} and other three examples, i.e., Corollary \ref{cor1}, Corollary \ref{cor2} and Corollary \ref{cor3}, which are the main results in the papers \cite{BerP21} and \cite{BerP19c}.
\smallskip

5. Example \ref{ex2} also shows that the class of enriched almost contractions is independent of both the class of nonexpansive mappings and of quasi-nonexpansive mappings. This is essentially different from Banach contractions, Kannan contractions, Chatterjea contractions, Zamfirescu contractions etc., which are all quasi-nonexpansive mappings.
\smallskip

6. We also presented an existence and convergence theorem for a generic  variational inequality problem, in Section \ref{s3}. 

7. The main merits of the paper are at least the following: a) it introduces a very large class of contractive type mappings for which the (non unique) fixed point can be obtained in a constructive way; b) it unifies in a single general concept various classes of contractive mappings possessing a unique fixed point (that can be approximated either by the Picard or by the  Krasnoselskij iteration) and some classes of mappings with non unique or even infinitely many fixed points (providing a constructive method of approximating these fixed points also).

8. Similar results to the ones stated in the present paper could be obtained by applying the technique of enriching contractive mappings for other classes of mappings in \cite{Ber04},  \cite{BerP16}, \cite{Viet} and references therein.
\smallskip

We end this section by formulating the following 
\medskip

{\bf Open problem}.
The enriched almost contraction $T$ in Example \ref{ex2} has the following interesting property:

a) The restrictions of $T$ to the subsets $\left[0,\dfrac{2}{3}\right)$ and $\left[\dfrac{2}{3}, \dfrac{4}{3}\right]$  are enriched contractions (and also nonexpansive mappings);

b) $\left[0,\dfrac{2}{3}\right)$, $\left[\dfrac{2}{3},\dfrac{4}{3}\right]$ form a partition of  $\left[0,\dfrac{4}{3}\right]$;

The problem is then the following one: is it true that for any enriched almost contraction $T$ defined on a Banach space $X$, one can find a partition of $X$,
$$
X=\bigcup_{i\in I} X_i,\, X_i\cap X_j=\emptyset, i\neq j, 
$$
such that the restriction of $T$ to $X_i$ is an enriched contraction (nonexpansive mapping)?

As a motivation of this problem, we recall that for almost contractions (see \cite{BerP16}) and, in general, for weakly Picard operators (see \cite{Rus01}, \cite{RPP}), we have similar properties:

a) For any weakly Picard mapping defined on a set $X$, one can find a partition of $X$, $
X=\bigcup_{i\in I} X_i,\, X_i\cap X_j=\emptyset, i\neq j, 
$
such that the restriction of $T$ to $X_i$ is a Picard operator.

b)  For any almost contraction defined on a set $X$, one can find a partition of $X$, $
X=\bigcup_{i\in I} X_i,\, X_i\cap X_j=\emptyset, i\neq j, 
$
such that the restriction of $T$ to $X_i$ is a Banach contraction, see \cite{BerP16}.
\bigskip

\section*{Acknowledgments}
The first author's research was supported by the Department of Mathematics and Computer Science, Technical University of Cluj-Napoca, North University Centre at Baia Mare, through the Internal Research Grant No. 4/2021.

\end{document}